\documentclass[a4paper,twoside]{amsart}
\usepackage[hmargin={25mm,25mm},vmargin={25mm,25mm}]{geometry}
\bibstyle{plain}

\newtheorem{theorem}{Theorem}
\newtheorem{lemma}{Lemma}
\newtheorem{proposition}{Proposition}
\newtheorem{definition}{Definition}

\theoremstyle{remark}

\renewcommand{\phi}{\varphi}
\renewcommand{\epsilon}{\varepsilon}

\begin{document}
\title[On the stability of the existence of fixed points\dots]{On the stability of the existence of fixed points for the projection-iterative methods with relaxation}
\author{Andrzej Komisarski}
\address{Andrzej Komisarski\\
Department of Probability Theory and Statistics\\Faculty of Mathematics and Computer Science\\
University of \L\'od\'z\\ul.Banacha 22\\90-238 \L\'od\'z\\Poland}
\email{andkom@math.uni.lodz.pl}
\author{Adam Paszkiewicz}
\address{Adam Paszkiewicz\\
Department of Probability Theory and Statistics\\Faculty of Mathematics and Computer Science\\
University of \L\'od\'z\\ul.Banacha 22\\90-238 \L\'od\'z\\Poland}
\email{adampasz@math.uni.lodz.pl}
\thanks{This paper is partially supported by NCN Grant no. N N201 605840.}
\begin{abstract}
We consider an $\alpha$-relaxed projection $P_A^\alpha:H\to H$ given by $P_A^\alpha(x)=\alpha P_A(x)+(1-\alpha)x$
where $\alpha\in[0,1]$ and $P_A$ is the projection onto a non-empty, convex and closed subset $A$ of the real Hilbert space $H$.
We characterise all the sets $F\subset[0,1]$ such that 
for some non-empty, convex and closed subsets $A_1,A_2,\dots,A_k\subset H$
the composition $P_{A_k}^\alpha P_{A_{k-1}}^\alpha\dots P_{A_1}^\alpha$ has a fixed point iff $\alpha\in F$.
It proves, that if $\dim H\geq 3$ and $k\geq3$ then the class of the derscribed above sets $F$ of coefficients $\alpha$
is exactly the class of $F_\sigma$ subsets of $[0,1]$ containing $0$.
\end{abstract}

\maketitle

\footnote{Andrzej Komisarski, Faculty of Mathematics and Computer Science, University of \L\'od\'z, \texttt{andkom@math.uni.lodz.pl}}
\footnote{Adam Paszkiewicz, Faculty of Mathematics and Computer Science, University of \L\'od\'z, \texttt{adampasz@math.uni.lodz.pl}}

The theory of fixed points plays a great role in applications.
In particular, researchers investigated fixed points of compositions $P_{A_k}P_{A_{k-1}}\dots P_{A_1}$
of projections onto non-empty convex
subsets $A_1$,\dots, $A_k$ of the real Hilbert (or Euclidean) space. For example Bregman (\cite{Bregman})
finds points in the intersection $A_1\cap\dots\cap A_k$
using cyclic iterations of the form $x_{n+1}=P_{A_{i_n}}x_n$, where $(i_n)$ is the cyclic sequence $(1,2,\dots,k,1,2,\dots,k,\dots)$.
Bregman provides conditions which assure that the sequence $(x_n)$ converges (even in the case of other linear metric spaces).
The problem of the convergence of the iterative methods of this type is closely related to the existence of fixed points.
If $\dim H<\infty$ then the convergence of $(x_n)$ for every starting point is equivalent to the existence 
of a common fixed point of the projections $P_{A_i}$.

Moreover, if $\dim H<\infty$ and $P_{A_k}P_{A_{k-1}}\dots P_{A_1}$ has a fixed point then for every $x\in H$ the sequence
$((P_{A_k}P_{A_{k-1}}\dots P_{A_1})^nx)$ is convergent.
However, if $\dim H=\infty$ then the existence of a fixed point of the composition $P_{A_k}P_{A_{k-1}}\dots P_{A_1}$
does not imply the norm convergence of $((P_{A_k}P_{A_{k-1}}\dots P_{A_1})^nx)$ for every $x\in H$,
even if $A_1\cap\dots\cap A_k\neq\emptyset$ and $k=2$ (cf. remrkable examples in \cite{Hundal} and \cite{Kopecka}).
Despite these negative results the investigation of fixed points of compositions $P_kP_{k-1}\dots P_1$
is the natural first step in research of iterations $((P_kP_{k-1}\dots P_1)^nx)$,
where $P_1,\dots,P_k$ are generalisations of the projections $P_{A_1},\dots,P_{A_k}$.

One of possible generalisations of projections arise if we consider
the relaxation parameter which is commonly used in the iterative methods to control the
rate of the convrgence and the regularity of trajectories.
In the case of projections introducing the relaxation parameter $\alpha$ replaces
the projection $P_Ax=x+(P_Ax-x)$ with a map $x+\alpha(P_Ax-x)=\alpha P_A(x)+(1-\alpha)x$.
This leads to the following definition:

\begin{definition}
Let $H$ be a real Hilbert space, let $A\subset H$ be a non-empty, convex and closed subset of $H$
and let $\alpha\in[0,1]$. An \emph{$\alpha$-relaxed projection} (or \emph{$\alpha$-projection}) onto $A$
is the function $P_A^\alpha:H\to H$ given by
$$P_A^\alpha(x)=\alpha P_A(x)+(1-\alpha)x,$$
where by $P_A:H\to A$ we denote the projection onto $A$.
\end{definition}

In the paper we concentrate on $\alpha$-projections but other generalisations of projections had also been used.
For example De Pierro (\cite{DePierro}) consiered iterations of convex combinations of projections.

Recently, De Pierro and Cegielski (oral communication, \cite{Cegielski}) formulated the following interesting problem concerning fixed points:
Let $A_1,A_2,A_3$ be non-empty, convex and closed subsets of the Hilbert space $H$
and let $\alpha\in(0,1)$.
Is the existence of a fixed point of the composition $P_{A_3}P_{A_2}P_{A_1}$
equivalent to the existence of a fixed point of the composition $P_{A_3}^\alpha P_{A_2}^\alpha P_{A_1}^\alpha$?
The answer is negative. Moreover, we have the following general result:

\begin{theorem}\label{gtw}
Let $H$ be a Hilbert space, $\dim H\geq 3$, let $k\geq3$ be an integer and let $F\subset[0,1]$.
The following conditions are equivalent:
\begin{enumerate}
\item There exist non-empty, convex and closed subsets $A_1,A_2,\dots,A_k\subset H$ satisfying
$$F=\{\alpha\in[0,1]:P_{A_k}^\alpha P_{A_{k-1}}^\alpha\dots P_{A_1}^\alpha\text{ has a fixed point}\},$$
\item $0\in F$ and $F$ is an $F_\sigma$ subset of $[0,1]$.
\end{enumerate}
\end{theorem}

It can be shown that if $\dim H=1$ or $k=1$ then the only set $F$ satisfying (i) is $[0,1]$.
If $k=2$ and $\dim H\geq 2$ then two sets $F$ satisfy (i), namely  $\{0\}$ and $[0,1]$.
It proves, that if $H=\mathbb R^2$ and $k\geq 3$ then the class of sets $F$ satisfying (i)
depends on $k$ and its full characterization is still an open problem.

Theorem \ref{gtw} is an immediate consequence of the following two propositions

\begin{proposition}\label{propi}
If $F$ is an $F_\sigma$ subset of $[0,1]$, $0\in F$ and $k\geq3$ then
$$F=\{\alpha\in[0,1]:P_{A_k}^\alpha P_{A_{k-1}}^\alpha\dots P_{A_1}^\alpha\text{ has a fixed point}\},$$
for some non-empty, convex and closed subsets $A_1,A_2,\dots,A_k\subset \mathbb R^3$.
\end{proposition}

\begin{proposition}\label{propii}
If $A_1,A_2,\dots,A_k$ are non-empty, convex and closed subsets of a Hilbert space $H$
then for every $r>0$ the set
$$F_r=\{\alpha\in[0,1]:P_{A_k}^\alpha P_{A_{k-1}}^\alpha\dots P_{A_1}^\alpha\text{ has a fixed point $x$ satisfying }\|x\|\leq r\}$$
is closed in $[0,1]$.
\end{proposition}

\begin{proof}[Proof of Theorem \ref{gtw}]
To show that (1) implies (2) it is enough to observe
that $P_{A_k}^0 P_{A_{k-1}}^0\dots P_{A_1}^0$ is the identity (hence $0\in F$)
and that $F=\bigcup_{r\in\mathbb N} F_r$, where $F_r$'s are the closed sets defined in Proposition \ref{propii}.

Now, let $F$ be any $F_\sigma$ subset of $[0,1]$, $0\in F$ and let $k\geq3$. 
By Proposition \ref{propi} we have
$$F=\{\alpha\in[0,1]:P_{A_k}^\alpha P_{A_{k-1}}^\alpha\dots P_{A_1}^\alpha\text{ has a fixed point}\},$$
for some non-empty, convex and closed subsets $A_1,A_2,\dots,A_k\subset \mathbb R^3$.
Using any isometric embedding of $\mathbb R^3$ into $H$ we obtain that (2) implies (1).
\end{proof}

\section{Proof of Proposition \ref{propi}}

Proposition \ref{propi} is a consequence of the following lemma
\begin{lemma}\label{lemgl}
If $F$ is an $F_\sigma$ subset of $[0,1]$ and $0\in F$ then
for some non-empty, convex and closed sets $A_1$, $A_2$, $A_3\subset\mathbb R^3$ one has:
\begin{enumerate}
\item if $\alpha\in F$ and $\beta\in[0,1]$ then $P_{A_3}^\alpha P_{A_2}^\alpha P_{A_1}^\beta$ has a fixed point,
\item if $\alpha\in [0,1]\setminus F$ and $\beta\in(0,1]$ then $P_{A_3}^\alpha P_{A_2}^\alpha P_{A_1}^\beta$ has no fixed point.
\end{enumerate}
\end{lemma}
\begin{proof}[Proof of Proposition \ref{propi}]
Let $F$ be an $F_\sigma$ subset of $[0,1]$ and let $A_1$, $A_2$ and $A_3$ be given by Lemma \ref{lemgl}.
Then $\alpha\in F$ iff $P_{A_3}^\alpha P_{A_2}^\alpha(P_{A_1}^\alpha)^{k-2}=P_{A_3}^\alpha P_{A_2}^\alpha P_{A_1}^{1-(1-\alpha)^{k-2}}$
has a fixed point (we put $\beta=1-(1-\alpha)^{k-2}$).
\end{proof}

The sets $A_1$, $A_2$ and $A_3$ demanded in Lemma \ref{lemgl} will be defined as
$A_1=\{(x,y,z):z\geq0,\ y\geq f(x,z)\}$, $A_2=\{(1,0,z):z\geq0\}$, $A_3=\{(0,0,z):z\geq0\}$
for some continuous convex function $f:\mathbb R\times[0,\infty)\to\mathbb R$.
The construction of the function $f$ will be the main part of the proof.
In particular we will use an auxiliary function $\phi$ defined by the following lemma.

\begin{lemma}\label{lempla}
Let $B_1=\{(x,y):y\geq x^2\}$,
$B_2=\{(1,0)\}$, $B_3=\{(0,0)\}$ be subsets of $\mathbb R^2$.
Then for every $\alpha,\beta\in(0,1]$ the composition $P_{B_3}^\alpha P_{B_2}^\alpha P_{B_1}^\beta$
has a unique fixed point $\mathbf u_{\alpha,\beta}$.
Moreover, there exists a decreasing and continuous function $\phi:(0,1]\to[0,1]$ such that
$P_{B_1}(\mathbf u_{\alpha,\beta})=(\phi(\alpha),\phi(\alpha)^2)$ for every $\alpha,\beta\in(0,1]$.
\end{lemma}
\begin{proof}
The existence and the uniqueness of the fixed point follows by the Banach fixed point theorem
for the contraction $P_{B_3}^\alpha P_{B_2}^\alpha P_{B_1}^\beta$.
Let us denote $\mathbf x_{\alpha,\beta}=(x_{\alpha,\beta},x_{\alpha,\beta}^2)=P_{B_1}(\mathbf u_{\alpha,\beta})$.
Then
\begin{equation*}
\begin{split}
\mathbf u_{\alpha,\beta}&=P_{B_3}^\alpha P_{B_2}^\alpha P_{B_1}^\beta(\mathbf u_{\alpha,\beta})
=P_{B_3}^\alpha P_{B_2}^\alpha((1-\beta)\mathbf u_{\alpha,\beta}+\beta\mathbf x_{\alpha,\beta})\\
&=(1-\alpha)^2(1-\beta)\cdot\mathbf u_{\alpha,\beta}+(1-\alpha)^2\beta\cdot\mathbf x_{\alpha,\beta}+(1-\alpha)\alpha\cdot(1,0)+\alpha\cdot(0,0),
\end{split}
\end{equation*}
hence
\begin{equation}\label{rowlem}
\frac{1-(1-\alpha)^2(1-\beta)}\alpha\cdot(\mathbf u_{\alpha,\beta}-\mathbf x_{\alpha,\beta})=(\alpha-2)\cdot\mathbf x_{\alpha,\beta}+(1-\alpha)\cdot(1,0).
\end{equation}
By $\mathbf x_{\alpha,\beta}=P_{B_1}(\mathbf u_{\alpha,\beta})$ it follows that $\mathbf u_{\alpha,\beta}-\mathbf x_{\alpha,\beta}$
is orthogonal to the tangent to $B_1$ at $\mathbf x_{\alpha,\beta}$, hence
$(\mathbf u_{\alpha,\beta}-\mathbf x_{\alpha,\beta})\perp(1,2x_{\alpha,\beta})$.
From \eqref{rowlem} we obtain
$$((\alpha-2)x_{\alpha,\beta}+(1-\alpha),\ (\alpha-2)x_{\alpha,\beta}^2)\ \cdot\ (1,\ 2x_{\alpha,\beta})=0,$$
which is equivalent to
$$2x_{\alpha,\beta}^3+x_{\alpha,\beta}=\frac{1-\alpha}{2-\alpha}.$$
Since the function $\psi(\alpha)=\frac{1-\alpha}{2-\alpha}$ is decreasing and continuous on $(0,1]$ and
the function $\chi(x)=2x^3+x$ is increasing and continuous on $\mathbb R$ and $\psi((0,1])=[0,\tfrac12)\subset\chi([0,1])$,
we obtain that $x_{\alpha,\beta}=\chi^{-1}(\psi(\alpha))\in[0,1]$ does not depend on $\beta$ and it is the decreasing
and continuous function of $\alpha$. We put $\phi(\alpha):=x_{\alpha,\beta}$.
\end{proof}
Letting $\phi(0)=\lim_{\alpha\to 0}\phi(\alpha)$ we extend $\phi$ to continuous and decreasing $\phi:[0,1]\to[0,1]$.

We pass to the construction of the function $f$ for a given
$F_\sigma$ subset $F\subset[0,1]$ satisfying $0\in F$.
We have $F=\bigcup_{n=1}^\infty F_n$ for some closed sets
$F_1\subset F_2\subset\dots\subset[0,1]$.
Let $E_n=(\mathbb R\setminus(-1,2))\cup\phi(F_n)$ for $n=1,2,\dots$.
(Here the interval $(-1,2)$ may be replaced by any open and bounded set containing $\phi([0,1])$.)
The sets $E_n$ are closed, $\inf E_n=-\infty$ and $\sup E_n=\infty$, hence the functions
$a_n,b_n:\mathbb R\to\mathbb R$ given by
$$a_n(x)=\max(E_n\cap(-\infty,x])\qquad\text{and}\qquad b_n(x)=\min(E_n\cap[x,\infty))$$
are well defined. Note, that if $x\in E_n$ then $a_n(x)=b_n(x)=x$. Otherwise, $(a_n(x),b_n(x))$
is the connected component of $\mathbb R\setminus E_n$ containing $x$. We define
\begin{equation}\label{deff}
f(x,z)=x^2+\sum_{n=1}^\infty c_ng_n(x)h_n(z),
\end{equation}
where
$$g_n(x)=(x-a_n(x))^3(b_n(x)-x)^3\qquad\text{and}\qquad h_n(z)=(n-z)_+^3=\begin{cases}(n-z)^3&\text{for }z\in[0,n]\\0&\text{for }z>n\end{cases}$$
and $(c_n)$ is any sequence with positive terms satisfying $\sum_{n=1}^\infty\frac{81^2}6n^3c_n<1$.

\begin{lemma}\label{lemmaf}
The function $f$ defined by \eqref{deff} satisfies the following conditions:
\begin{enumerate}
\item $f(x,z)\geq x^2$ for every $x\in\mathbb R$ and $z\geq0$,
\item $f\in C^2$ and $f$ is convex,
\item For every $x\in\mathbb R$ one has: $x\in\bigcup_{n=1}^\infty E_n\ \Leftrightarrow\ \exists_{z\geq0}\ f(x,z)=x^2$,
\item For every $x\in\mathbb R$ and $z\geq0$ if $f(x,z)>x^2$ then $\frac{\partial f}{\partial z}(x,z)<0$.
\end{enumerate}
\end{lemma}

\begin{proof} We have
$$g'_n(x)=3(x-a_n(x))^2(b_n(x)-x)^2(a_n(x)+b_n(x)-2x),\qquad h'_n(z)=-3((n-z)_+)^2,$$
$$g''_n(x)=6(x-a_n(x))(b_n(x)-x)[(a_n(x)+b_n(x)-2x)^2-(x-a_n(x))(b_n(x)-x)]\qquad\text{and}\qquad h''_n(z)=6(n-z)_+.$$
For every $x\in\mathbb R$ and $z>0$ one has
$$|g_n(x)|\leq 3^6,\qquad |g'_n(x)|\leq 3^6,\qquad |g''_n(x)|\leq6\cdot3^4,$$
$$|h_n(z)|\leq n^3,\qquad |h'_n(z)|\leq 3n^2\quad\text{and}\quad h''_n(z)\leq 6n.$$
It follows, that the series \eqref{deff} is uniformly convergent.
Moreover, if we try to calculate the first and the second order derivatives of $f$
by the formal differentiation of the series \eqref{deff} term by term
then we obtain a uniformly convergent series with continuous terms.
It follows that $f$ is well defined and $f\in C^2$.
Since $g_n(x),h_n(z)\geq0$ we get (i).

We will check the convexity of $f$ by showing that
the Hessian matrix $H(f)(x,z)$ is positive semidefinite for every $x\in\mathbb R$ and $z>0$.
$$H(f)(x,z)=
\begin{pmatrix}
2+\sum_{n=1}^\infty c_ng''_n(x)h_n(z)
&\sum_{n=1}^\infty c_ng'_n(x)h'_n(z)\\
\sum_{n=1}^\infty c_ng'_n(x)h'_n(z)
&\sum_{n=1}^\infty c_ng_n(x)h''_n(z)
\end{pmatrix}.$$
Clearly $\sum_{n=1}^\infty c_ng_n(x)h''_n(z)\geq0$ and
$$2+\sum_{n=1}^\infty c_ng''_n(x)h_n(z)\geq 2-\sum_{n=1}^\infty c_n\cdot6\cdot3^4\cdot n^3>2-\sum_{n=1}^\infty\frac{81^2}6n^3c_n>1.$$
Moreover,
\begin{equation*}
\begin{split}
\det H(f)(x,z)=&\left(2+\sum_{n=1}^\infty c_ng''_n(x)h_n(z)\right)\left(\sum_{n=1}^\infty c_ng_n(x)h''_n(z)\right)-\left(\sum_{n=1}^\infty c_ng'_n(x)h'_n(z)\right)^2\\
\geq&1\cdot\sum_{n=1}^\infty c_ng_n(x)h''_n(z)-\left(\sum_{n=1}^\infty c_ng'_n(x)h'_n(z)\right)^2\\
\geq&\sum_{n=1}^\infty \frac{81^2}6n^3c_n\cdot\sum_{n=1}^\infty 6c_n(x-a_n(x))^3(b_n(x)-x)^3(n-z)_+\\
&\qquad\qquad-\left(\sum_{n=1}^\infty 9c_n(x-a_n(x))^2(b_n(x)-x)^2(a_n(x)+b_n(x)-2x)((n-z)_+)^2\right)^2\\
\geq&\sum_{n=1}^\infty\frac{81^2}6n^3c_n\cdot\sum_{n=1}^\infty 6c_n(x-a_n(x))^3(b_n(x)-x)^3(n-z)_+\\
&\qquad\qquad-\left(\sum_{n=1}^\infty 81n^{\tfrac32}c_n(x-a_n(x))^{\tfrac32}(b_n(x)-x)^{\tfrac32}((n-z)_+)^{\tfrac12}\right)^2\geq0
\end{split}
\end{equation*}
In the above we used inequalities $0\leq x-a_n(x)\leq3$, $0\leq b_n(x)-x\leq3$ (hence $|a_n(x)+b_n(x)-2x|\leq3$) and, finally, the Schwartz inequality.
We obtained that the Hessian matrix $H(f)(x,z)$ is positive semidefinite, hence we have (ii).

Now, we will show (iii). 
If $x\in\bigcup_{n=1}^\infty E_n$ then $x\in E_{n_0}$ for some $n_0$ and (since $E_1\subset E_2\subset\dots$)
$x\in E_n$ for every $n\geq n_0$. If $x\in E_n$ then $g_n(x)=0$ (by the definition of $g_n$).
Consequently, for every $z>n_0$ we have
$$f(x,z)=x^2+\sum_{n<n_0}c_ng_n(x)\cdot0+\sum_{n\geq n_0} c_n\cdot0\cdot h_n(z)=x^2.$$
If $x\notin\bigcup_{n=1}^\infty E_n$ then $g_n(x)>0$ for every $n$.
Let $z\geq0$. Then $z<n_0$ (hence $h_{n_0}(z)>0$) for some $n_0$ and we have
$$f(x,z)\geq x^2+c_{n_0}g_{n_0}(x)h_{n_0}(z)>x^2.$$

Finally, we will show (iv). First observe that for every $n$, $x$ and $z$ we have $\frac{\partial c_ng_n(x)h_n(z)}{\partial z}(x,z)\leq0$.
It follows that if $f(x,z)>x^2$ then $g_{n_0}(x)>0$ and $h_{n_0}(z)>0$ for some $n_0$ and
$\frac{\partial f}{\partial z}(x,z)\leq\frac{\partial c_{n_0}g_{n_0}(x)h_{n_0}(z)}{\partial z}(x,z)<0$
\end{proof}

\begin{proof}[Proof of Lemma \ref{lemgl}]
Let $F$ be an $F_\sigma$ subset of $[0,1]$ satisfying $0\in F$. We define the sets $A_1,A_2,A_3\subset \mathbb R^3$ as follows:
\begin{equation*}
\begin{split}
A_1&=\{(x,y,z):z\geq0,\ y\geq f(x,z)\},\\
A_2&=\{(1,0,z):z\geq0\},\\
A_3&=\{(0,0,z):z\geq0\},
\end{split}
\end{equation*}
where the continuous convex function $f:\mathbb R\times[0,\infty)\to\mathbb R$ is defined by \eqref{deff}.
Moreover, let $A'_1=\{(x,y,z):z\geq0,\ y\geq x^2\}$.

If $\alpha=0$ and $\beta\in[0,1]$ then $P_{A_3}^\alpha P_{A_2}^\alpha P_{A_1}^\beta=P_{A_1}^\beta$
and every $\mathbf u\in A_1$ is a fixed point of $P_{A_3}^\alpha P_{A_2}^\alpha P_{A_1}^\beta$.

Let $\alpha\in F\setminus\{0\}$ and $\beta\in[0,1]$. Then $\phi(\alpha)\in\bigcup_{n=1}^\infty E_n$
and by Lemma \ref{lemmaf} (iii) there exists $z\geq0$ such that $f(\phi(\alpha),z)=\phi(\alpha)^2$,
i.e. $(\phi(\alpha),\phi(\alpha)^2,z)\in A_1$.
Let $(u,v)=\mathbf u_{\alpha,\beta}$ be the fixed point of $P_{B_3}^\alpha P_{B_2}^\alpha P_{B_1}^\beta$ given in Lemma \ref{lempla}.
Then $(u,v,z)$ is a fixed point of $P_{A_3}^\alpha P_{A_2}^\alpha P_{A'_1}^\beta$
(because $A'_1=B_1\times[0,\infty)$, $A_2=B_2\times[0,\infty)$, $A_3=B_3\times[0,\infty)$ and $z\geq0$).
Moreover, $A_1\subset A_1'$ (by Lemma \ref{lemmaf} (i)) and $P_{A'_1}(u,v,z)=(\phi(\alpha),\phi(\alpha)^2,z)\in A_1$.
If follows that $P_{A_1}(u,v,z)=P_{A'_1}(u,v,z)$, hence $(u,v,z)$ is a fixed point of $P_{A_3}^\alpha P_{A_2}^\alpha P_{A_1}^\beta$.

Finally, let $\alpha\in[0,1]\setminus F$ and let $\beta\in(0,1]$.
Assume, aiming at a contradiction, that $(u,v,z)$ is the fixed point of $P_{A_3}^\alpha P_{A_2}^\alpha P_{A_1}^\beta$.
Since $\alpha,\beta>0$, we obtain that $(u,v,z)$ is ouside the set $A_1$ and $z\geq0$
which means that $(u_1,v_1,z_1):=P_{A_1}(u,v,z)$ satisfies $v_1=f(u_1,z_1)$.

If $v_1=f(u_1,z_1)>u_1^2$ then by Lemma~\ref{lemmaf} (iv) we have
$\frac{\partial f}{\partial z}(u_1,z_1)<0$. Consequently $z_1>z$.
It follows that the last coordinate of $P_{A_3}^\alpha P_{A_2}^\alpha P_{A_1}^\beta(u,v,z)$
which is equal to the last coordinate of $P_{A_1}^\beta(u,v,z)=\beta z_1+(1-\beta)z$ is greater than $z$.
We obtained a contradiction, since $(u,v,z)$ is a fixed point of $P_{A_3}^\alpha P_{A_2}^\alpha P_{A_1}^\beta$.

Thus $v_1=f(u_1,z_1)=u_1^2$, hence $(u_1,v_1,z_1)$ is located at the bounadaries of both $A_1$ and $A'_1$.
Since for both $A_1$ and $A'_1$ there exist tangent planes at $(u_1,v_1,z_1)$ and $A_1\subset A_1'$,
it follows that these two planes are equal. Consequently, $P_{A'_1}(u,v,z)=P_{A_1}(u,v,z)=(u_1,v_1,z_1)$,
hence $(u,v,z)$ is a fixed point of $P_{A_3}^\alpha P_{A_2}^\alpha P_{A'_1}^\beta$.
By Lemma \ref{lempla} we obtain $(u_1,v_1,z_1)=(u_1,f(u_1,z_1),z_1)=(\phi(\alpha),\phi(\alpha)^2,z_1)$,
in particular $f(\phi(\alpha),z_1)=\phi(\alpha)^2$.
Finally, by Lemma \ref{lemmaf} (iii) we get $\phi(\alpha)\in\bigcup_{n=1}^\infty E_n$, hence $\alpha\in F$.
We got the contradiction. Hence $P_{A_3}^\alpha P_{A_2}^\alpha P_{A_1}^\beta$ has no fixed point, as required.
\end{proof}

\section{Proof of Proposition \ref{propii}}

Let $P^\alpha:=P_{A_k}^\alpha P_{A_{k-1}}^\alpha\dots P_{A_1}^\alpha$.

If $\dim H<\infty$ then Proposition \ref{propii} is a consequence of the compactess
of the closed balls in $H$. Indeed, let $\alpha_1,\alpha_2,\dots\in F_r$ and let $\alpha_n\to\alpha_0$.
Then for every $n$ we have $P^{\alpha_n} x_n=x_n$
for some $x_n\in H$ satisfying $\|x_n\|\leq r$. Considering subsequences of $(x_n)$ and $(\alpha_n)$
we may assume that $(x_n)$ is convergent to some $x_0\in H$ with $\|x_0\|\leq r$. Using the continuity of the function
$(\alpha,x)\mapsto P^\alpha x$ we obtain
$$P^{\alpha_0}x_0=\lim_{n\to\infty} P^{\alpha_n}x_n=\lim_{n\to\infty}x_n=x_0.$$
It follows that $\alpha_0\in F_r$ hence $F_r$ is closed.

If $\dim H=\infty$ then the ball in $H$ is not compact and the above reasoning does not work.
One idea is to consider the weak topology on $H$ (instead of the norm topology).
Unfortunately it still does not work, because the projection onto a closed convex set in $H$
does not need to be weakly continuous. For these reasons if $\dim H=\infty$ then the proof is more complicated.
The idea is as follows: Using the compactness of a closed ball in the weak topology
we will find $x_0$ which is a condensation point in the weak topology of the defined above sequence $(x_n)$ and then we will construct
a sequence $(u_M)$ satisfying $\|u_M-x_0\|\to0$ and $\|P^{\alpha_0}(u_M)-u_M\|\to0$ for $M\to\infty$. Then, by the continuity
of $x\mapsto P^{\alpha_0} x$ in the norm topology we obtain $P^{\alpha_0}(x_0)=x_0$.

\begin{lemma}\label{lemat}
Let $M\in\mathbb N$ and let $(y_i^n)_{i=1,\dots,M}^{n\in\mathbb N}$ and $(y^n)_{n\in\mathbb N}$ be systems of elements of $H$ satisfying:
\begin{enumerate}
\item $\lim_{n\to\infty}\|y_i^n\|=1$ for $i=1,\dots,M$,
\item $\lim_{n\to\infty}(y_i^n,y_j^n)=0$ for $i\neq j$, $i,j=1,\dots,M$,
\item $\limsup_{n\to\infty}\|y^n-y_i^n\|^2\leq\frac{M-1}M$ for $i=1,\dots,M$.
\end{enumerate}
Then $\lim_{n\to\infty}\|y^n-\frac{y_1^n+\dots+y_M^n}M\|=0$.
\end{lemma}
\begin{proof}
We have $y^n=z^n+\sum_{i=1}^M\alpha_i^ny_i^n$
for some $\alpha_i^n\in\mathbb R$ and $z^n\in H$ with $z_n\perp y_i^n$ for $i=1,\dots,M$.
Then, by (i), (ii) and (iii), for large enough $n$ one has
\begin{equation*}
\begin{split}
4&>(\|y^n-y_i^n\|+\|y_i^n\|)^2\geq\|y^n\|^2\geq\left\|\sum_{i=1}^M\alpha_i^ny_i^n\right\|
=\sum_{i=1}^M(\alpha_i^n)^2\|y_i^n\|^2+\sum_{i\neq j}\alpha_i^n\alpha_j^n(y_i^n,y_j^n)\\
&\geq\sum_{i=1}^M(\alpha_i^n)^2\|y_i^n\|^2-\sum_{i\neq j}\frac{(\alpha_i^n)^2+(\alpha_j^n)^2}2|(y_i^n,y_j^n)|
=\sum_{i=1}^M(\alpha_i^n)^2\left(\|y_i^n\|^2-\sum_{j\neq i}|(y_i^n,y_j^n)|\right)
\geq\frac12\sum_{i=1}^M(\alpha_i^n)^2.
\end{split}
\end{equation*}
It follows that all $\alpha_i^n$'s are bounded.
Moreover, for every $l=1,\dots,M$ one has
$$\|y^n-y_l^n\|^2=\|z^n\|^2+\sum_{i\neq l}(\alpha_i^n)^2\|y_i^n\|^2+(\alpha_l^n-1)^2\|y_l^n\|^2
+\sum_{i\neq j,\ i,j\neq l}\alpha_i^n\alpha_j^n(y_i^n,y_j^n)+\sum_{i\neq l}\alpha_i^n(\alpha_l^n-1)(y_i^n,y_l^n),$$
hence (taking $\limsup$ in the above and using (i), (ii) and (iii) and the boundedness of $\alpha_i^n$'s)
$$\limsup_{n\to\infty}\left(\|z^n\|^2+\sum_{i=1}^M(\alpha_i^n)^2-2\alpha_l^n+1\right)\leq\frac{M-1}M.$$
Summing the above inequalities with $l=1,\dots,M$ we obtain
$$\limsup_{n\to\infty}\left(M\|z^n\|^2+M\sum_{i=1}^M(\alpha_i^n)^2-2\sum_{l=1}^M\alpha_l^n+M\right)\leq M-1,$$
which is equivalent to
$$\limsup_{n\to\infty}\left(M\|z^n\|^2+M\sum_{i=1}^M\left(\alpha_i^n-\frac1M\right)^2\right)\leq0.$$
It follows that $\lim_{n\to\infty}\|z^n\|=0$ and $\lim_{n\to\infty}\alpha_i^n=\tfrac1M$ for $i=1,\dots,M$,
hence
$$\left\|y_n-\frac{y_1^n+\dots+y_M^n}M\right\|\leq\|z_n\|+\sum_{i=1}^M|\alpha_i^n-\tfrac1M|\|y_i^n\|\to0.$$
\end{proof}

We are ready to proove Proposition \ref{propii} in the general case.
Let $\alpha_1,\alpha_2,\dots\in F_r$ and let $\alpha_n\to\alpha_0$.
For every $n$ let $x_n\in H$ satisfy $P^{\alpha_n} x_n=x_n$ and $\|x_n\|\leq r$.
Considering subsequences of $(x_n)$ and $(\alpha_n)$
we may assume that $(x_n)$ is weakly convergent to some $x_0\in H$ with $\|x_0\|\leq r$.
Again, considering subsequences of $(x_n)$ and $(\alpha_n)$ we may assume that:
\begin{itemize}
\item $\lim_{n\to\infty}\|x_n-x_0\|=0$,

or
\item $\lim_{n\to\infty}\|x_n-x_0\|=\lambda$ for some $\lambda>0$ and $(x_n-x_0,x_m-x_0)\to0$ when $n,m\to 0$.
\end{itemize}

If $\lim_{n\to\infty}\|x_n-x_0\|=0$ then (similarly as in finite dimensional case) by the continuity of the function
$(\alpha,x)\mapsto P^\alpha x$ we obtain
$$P^{\alpha_0}x_0=\lim_{n\to\infty} P^{\alpha_n}x_n=\lim_{n\to\infty}x_n=x_0$$
and we are done.

Otherwise (if $\lim_{n\to\infty}\|x_n-x_0\|=\lambda$ for some $\lambda>0$ and $(x_n-x_0,x_m-x_0)\to0$ when $n,m\to 0$)
we proceed as follows: For any fixed $M\in\mathbb N$ we define
$$y_i^n=\tfrac1\lambda\left(P^{\alpha_0}(x_{n+i})-x_0\right)\qquad\text{for }n\in\mathbb N\text{ and }i=1,\dots,M,$$
$$y^n=\tfrac1\lambda\left(P^{\alpha_0}\left(\frac{x_{n+1}+\dots+x_{n+M}}M\right)-x_0\right)\qquad\text{for }n\in\mathbb N.$$
We will check that $(y_i^n)_{i=1,\dots,M}^{n\in\mathbb N}$ and $(y^n)_{n\in\mathbb N}$ satisfy the assumptions of Lemma \ref{lemat}.

(i). We have
$$\|y_i^n\|=\left\|\tfrac1\lambda\left(P^{\alpha_0}(x_{n+i})-x_0\right)\right\|
=\left\|\frac{x_{n+i}-x_0}\lambda+\frac{P^{\alpha_0}(x_{n+i})-x_{n+i}}\lambda\right\|,$$
which yields (i), because
$\left\|\frac{x_{n+i}-x_0}\lambda\right\|\to1$
and $\|P^{\alpha_0}(x_{n+i})-x_{n+i}\|=\|P^{\alpha_0}(x_{n+i})-P^{\alpha_{n+i}}(x_{n+i})\|\to0$.

Similarly (by $(x_{n+i}-x_0,x_{n+j}-x_0)\to0$ for $i\neq j$ and $n\to\infty$) we obtain (ii).

(iii). We have
\begin{equation*}
\begin{split}
\|y^n-y_i^n\|^2
&=\left\|\frac{P^{\alpha_0}\left(\frac{x_{n+1}+\dots+x_{n+M}}M\right)-P^{\alpha_0}(x_{n+i})}\lambda\right\|^2
\leq\frac1{\lambda^2}\left\|\frac{x_{n+1}+\dots+x_{n+M}}M-x_{n+i}\right\|^2\\
&=\frac1{\lambda^2}\left\|\sum_{j\neq i}\frac1M(x_{n+j}-x_0)-\frac{M-1}M(x_{n+i}-x_0)\right\|^2
\end{split}
\end{equation*}
and by 
$\|x_{n+i}-x_0\|\to\lambda$ and $(x_{n+i}-x_0,x_{n+j}-x_0)\to0$ for $i\neq j$ and $n\to\infty$ we obtain
$$\limsup_{n\to\infty}\|y^n-y_i^n\|^2
\leq\frac1{\lambda^2}\left(\sum_{j\neq i}\frac1{M^2}\lambda^2+\frac{(M-1)^2}{M^2}\lambda^2\right)=\frac{M-1}M.$$

By Lemma \ref{lemat} we obtain
\begin{equation}\label{nierown}
\lim_{n\to\infty}\left\|P^{\alpha_0}\left(\frac{x_{n+1}+\dots+x_{n+M}}M\right)-\frac{P^{\alpha_0}(x_{n+1})+\dots+P^{\alpha_0}(x_{n+M})}M\right\|=0$$
which (by $\|P^{\alpha_0}(x_{n+i})-x_{n+i}\|\to0$) is equivalent to
$$\lim_{n\to\infty}\left\|P^{\alpha_0}\left(\frac{x_{n+1}+\dots+x_{n+M}}M\right)-\frac{x_{n+1}+\dots+x_{n+M}}M\right\|=0.
\end{equation}

On the other hand, for large enough $n$ we have
$$\left\|\frac{x_{n+1}+\dots+x_{n+M}}M-x_0\right\|^2=\frac1{M^2}\left(\sum_{i=1}^M\|x_{n+i}-x_0\|^2+\sum_{i\neq j}(x_{n+i}-x_0,x_{n+j}-x_0)\right)
<\frac{2\lambda^2}M.$$
By the above inequality and by \eqref{nierown} it follows that choosing large enough $n$
and letting
$$u_M:=\frac{x_{n+1}+\dots+x_{n+M}}M$$
we have $\|P^{\alpha_0}(u_M)-u_M\|<\frac1M$ and $\|u_M-x_0\|<\lambda\sqrt{\frac2M}$.

We constructed the sequence $(u_M)$ satisfying $\|u_M-x_0\|\to0$ and $\|P^{\alpha_0}(u_M)-u_M\|\to0$ for $M\to\infty$.
Hence
$P^{\alpha_0}(x_0)=\lim_{M\to\infty}P^{\alpha_0}(u_M)=\lim_{M\to\infty}u_M=x_0$.
Thus $\alpha_0\in F_r$ and $F_r$ is closed.


\end{document}